\documentclass[11pt,reqno,letterpaper]{amsart}

\usepackage[T1]{fontenc}
\usepackage[utf8]{inputenc}
\usepackage{amssymb,amsmath,amsthm,mathrsfs,mathtools}
\usepackage{doi}
\usepackage{bookmark}
\usepackage{hyperref}
\hypersetup{pdfstartview={FitH},hidelinks}

\numberwithin{equation}{section}

\newtheorem{theorem}{Theorem}
\newtheorem{lemma}[theorem]{Lemma}
\newtheorem{proposition}[theorem]{Proposition}
\newtheorem{corollary}[theorem]{Corollary}

\theoremstyle{definition}

\newtheorem{example}[theorem]{Example}

\renewcommand{\leq}{\leqslant}

\renewcommand{\le}{\leqslant}
\renewcommand{\ge}{\geqslant}

\newcommand{\C}{\mathbb{C}}
\newcommand{\N}{\mathbb{N}}
\newcommand{\R}{\mathbb{R}}
\newcommand{\T}{\mathbb{T}}
\newcommand{\Z}{\mathbb{Z}}
\newcommand{\cube}{\{0,1\}}
\newcommand{\SU}{\mathrm{SU}}
\newcommand{\dd}{\,\mathrm{d}}
\newcommand{\wt}{\widetilde}

\begin{document}

\title[Hausdorff--Young inequalities for \texorpdfstring{$\SU(1,1)$}{SU(1,1)} Fourier products]{Hausdorff--Young inequalities for \texorpdfstring{$\SU(1,1)$}{SU(1,1)} Fourier products indexed by binary cubes}

\author[T. Crmari\'{c}]{Ton\'{c}i Crmari\'{c}}
\address[T.\,C.]{Department of Mathematics, Faculty of Science, University of Split, Ru\dj{}era Bo\v{s}kovi\'{c}a 33, 21000 Split, Croatia}
\email{tcrmaric@pmfst.hr}

\author[J. Rup\v{c}i\'c]{Jelena Rup\v{c}i\'c}
\address[J.\,R.]{Faculty of Transport and Traffic Sciences, University of Zagreb, Vukeli\'{c}eva 4, 10000 Zagreb, Croatia}
\email{jrupcic@fpz.unizg.hr}



\begin{abstract}
We study nonlinear logarithmic Hausdorff–Young inequalities with constant \(1\) for \(\mathrm{SU}(1,1)\)-valued Fourier products indexed by binary cubes. The inequalities are proved for the optimal range of exponents, extending into the region \(1/p+1/q<1\). The proof combines majorization, a known sharp two-point inequality, and tree iteration. We also obtain the corresponding chain inequality and nonlinear bounds for generating functions of alternating chains, recovering the known sharp bounds for additive energies on the binary cube as the amplitudes vanish.
\end{abstract}

\maketitle

\section{Introduction}

A discrete $\SU(1,1)$-valued nonlinear Fourier transform is obtained from a finitely supported sequence $F=(F_n)_{n\in\Z}$ of complex numbers in the open unit disk.  Set
\[
A_n=\frac{1}{(1-|F_n|^2)^{1/2}},\quad
B_n=\frac{F_n}{(1-|F_n|^2)^{1/2}},
\]
so that for each $n$ we have $A_n>0$, $B_n\in\C$ and $|A_n|^2-|B_n|^2=1$. We form the ordered product
\begin{equation}\label{eq:dnft}
\prod_{n\in\Z}
\begin{pmatrix}
A_n & B_n z^n\\
\overline{B_n}z^{-n} & \overline{A_n}
\end{pmatrix}
=
\begin{pmatrix}
a(z)&b(z)\\
\overline{b(z)}&\overline{a(z)}
\end{pmatrix},
\quad z=e^{2\pi i t}\in \mathbb S^1,
\end{equation}
where the factors are ordered from left to right as $n$ increases. Since $F$ is finitely supported, only finitely many factors in the product are different from the identity matrix. We identify functions on $\mathbb S^1:=\{z\in\C:|z|=1\}$ with
functions on $\T$ through the parametrization
$z=e^{2\pi i t}$. Throughout the paper, $\T=\R/\Z$ is equipped with normalized Haar measure.  All logarithms are natural, except for $\log_2$.  For each fixed $z\in\mathbb S^1$, every factor in \eqref{eq:dnft}, and hence the product itself, belongs to the group
\[ \SU(1,1)=
\left\{
\begin{pmatrix}
A&B\\
\overline{B}&\overline{A}
\end{pmatrix}
:
A,B\in\C,\ |A|^2-|B|^2=1
\right\}. \]
This model appears naturally in the theory of orthogonal polynomials on the unit circle; see Simon's monographs \cite{Simon1,Simon2} and the lecture notes of Tao and Thiele \cite{TT03}.  The transform satisfies the nonlinear Parseval identity, which goes back to Verblunsky \cite{Verblunsky36},
\begin{equation*}
\int_{\T}\log|a(t)|^2\dd t=\sum_{n\in\Z}\log|A_n|^2.
\end{equation*}
This identity motivates nonlinear Hausdorff--Young inequalities of the form
\begin{equation}\label{eq:nhy}
\left\|\bigl(\log|a(t)|^2\bigr)^{1/2}\right\|_{\textup{L}^q(\T)}
\le C_p
\left\|\bigl(\log|A_n|^2\bigr)^{1/2}\right\|_{\ell^p(\Z)}
\end{equation} for $1\leq p<2$ and its conjugate exponent $2<q\leq\infty$, i.e., $1/p+1/q=1$.

The continuous version of \eqref{eq:nhy} for Dirac scattering goes back to Christ and Kiselev \cite{CK01a,CK01b}.  Muscalu, Tao, and Thiele \cite{MTT03} raised the question of constants independent of the exponent for the scattering transform, and Kova\v{c} proved such bounds in a Cantor-group model \cite{K12}. Kova\v{c}, Oliveira e Silva, and one of the present authors obtained a sharp inequality for small potentials for the continuous transform \cite{KOR19} and asymptotically sharp discrete inequalities for $\SU(1,1)$-valued Fourier products \cite{KOR22}. In a recent preprint, Suragan proved the discrete inequality with the sharp constant 1 and obtained the corresponding continuous inequality \cite{Suragan2026}.

For variational estimates, see the work of Oberlin, Seeger, Tao, Thiele, and Wright \cite{OSTTW12}; in the discrete $\SU(1,1)$ setting, Oliveira e Silva proved a variational nonlinear Hausdorff--Young inequality \cite{OeS18}. Poltoratski proved a pointwise convergence theorem for the scattering data of continuous Dirac systems \cite{Poltoratski24}. In the discrete setting, one of the present authors studied convergence of lacunary $\SU(1,1)$ products \cite{Rupcic19}. Saksida developed an inverse theory for discrete nonlinear Fourier transforms \cite{Saksida22} and studied probabilistic and combinatorial aspects of nonlinear Fourier expansions \cite{Saksida24}. 

In this paper, we study a finite multiparameter setting of
$\SU(1,1)$-valued Fourier products. A family of $2^d$ matrices,
indexed by the binary cube $\{0,1\}^d$, is combined recursively
along a complete binary tree, with one independent torus variable
attached to each level. For real $q>2$, we write
\[
\binom{q}{q/2}
:=
\frac{\Gamma(q+1)}{\Gamma(q/2+1)^2}.
\]
We prove that, whenever $1<p<2<q<\infty$ and
\begin{equation}\label{eq:range}
\frac1p
\ge
\frac1q\log_2\binom{q}{q/2},
\end{equation}
the analogue of \eqref{eq:nhy} holds for the root matrix with constant \(1\), with the \(\textup L^q\)-norm taken over \(\mathbb T^d\) and the \(\ell^p\)-norm taken over the leaves. This
exponent range, which is optimal, contains the usual conjugate pairs and extends
beyond them. For conjugate pairs, the estimate also follows from Suragan’s discrete theorem \cite{Suragan2026}. Our proof does not use this result. A sparse choice of nonidentity matrices at the leaves gives the analogous estimate for chain products.

The proof reduces the estimate at each branching to a comparison between
the logarithmic expression associated with a two-factor $\SU(1,1)$ product and the squared modulus of a two-term trigonometric polynomial. The sharp two-point Hausdorff–Young inequality of Kovač, Shiraki, and one of the present authors \cite{CKS25} then applies.
Iteration along the tree, together with Minkowski's inequality,
yields the full multiparameter result. Linearizing this result as the amplitudes tend to zero gives the sharp Hausdorff–Young inequality on the binary cube from \cite{CKS25}, which proves that the exponent range
in \eqref{eq:range} is optimal. For the chain inequality, necessity
is already present in the two-factor case.

The same estimate also has combinatorial consequences. For subsets of the binary cube, the estimate gives nonlinear bounds for generating functions whose coefficients count alternating chains. Letting the amplitudes tend to zero recovers the sharp inequality for additive energy on the binary cube from \cite{CKS25}.

\section{Fourier products indexed by binary cubes}

For $s\in\R$, set
\[
D(s):=
\begin{pmatrix}
e^{\pi i s}&0\\
0&e^{-\pi i s}
\end{pmatrix} \in \SU(1,1).
\]
Then
\[
D(s)
\begin{pmatrix}
A&B\\
\overline{B}&\overline{A}
\end{pmatrix}
D(s)^{-1}
=
\begin{pmatrix}
A&Be^{2\pi i s}\\
\overline{B}e^{-2\pi i s}&\overline{A}
\end{pmatrix}.
\]
For words $\sigma$ and $\tau$ over the alphabet $\cube$, write $\sigma\tau$ for their concatenation. We use the convention $\cube^0=\{\emptyset\}$, where $\emptyset$
denotes the empty word.

\begin{theorem}\label{thm:cube}
Let $1<p<2<q<\infty$ satisfy \eqref{eq:range}.  Fix $d\in\N$, and for each $\omega\in\cube^d$ let
\[
G_\omega=\begin{pmatrix}A_\omega&B_\omega\\ \overline{B_\omega}&\overline{A_\omega}\end{pmatrix}\in\SU(1,1).
\]
Set
\[
P_\omega^{(0)}:=G_\omega,
\qquad \omega\in\cube^d.
\]
For $1\le n\le d$ and $\sigma\in\cube^{d-n}$, define recursively
\begin{equation}\label{eq:tree-recursion}
P_\sigma^{(n)}(t_1,\dots,t_n)
:=P_{\sigma0}^{(n-1)}(t_1,\dots,t_{n-1})D(t_n)P_{\sigma1}^{(n-1)}(t_1,\dots,t_{n-1})D(t_n)^{-1}.
\end{equation}
For $t=(t_1,\dots,t_n)\in\T^n$, each $P_\sigma^{(n)}(t)$ belongs to $\SU(1,1)$; 
write
\[
P_\sigma^{(n)}(t)=
\begin{pmatrix}
a_\sigma^{(n)}(t)&b_\sigma^{(n)}(t)\\
\overline{b_\sigma^{(n)}(t)}&\overline{a_\sigma^{(n)}(t)}
\end{pmatrix}.
\]
Then
\begin{equation}\label{eq:cube-main}
\left\|\bigl(\log|a_\emptyset^{(d)}|^2\bigr)^{1/2}\right\|_{\textup{L}^q(\T^d)}
\le
\left(\sum_{\omega\in\cube^d}\bigl(\log|A_\omega|^2\bigr)^{p/2}\right)^{1/p}.
\end{equation}
\end{theorem}

A sparse assignment of nonidentity matrices to the leaves yields the following
chain inequality.

\begin{corollary}\label{cor:chain}
Let $1<p<2<q<\infty$ satisfy \eqref{eq:range}. Let $N\in\N$, and
for $j=0,1,\dots,N$ let
\[
G_j=
\begin{pmatrix}
A_j&B_j\\
\overline{B_j}&\overline{A_j}
\end{pmatrix}
\in\SU(1,1).
\]
For $j=1,\dots,N$ and $s\in\T$ set
\[
\wt G_j(s):=D(s)G_jD(s)^{-1}.
\]
For $t=(t_1,\dots,t_N)\in\T^N$, define
\[
Q_N(t)
:=
G_0\wt G_1(t_1)\cdots\wt G_N(t_N)
=
\begin{pmatrix}
a_N(t)&*\\
*&*
\end{pmatrix}.
\]
Then
\begin{equation}\label{eq:chain}
\left\|
\bigl(\log|a_N|^2\bigr)^{1/2}
\right\|_{\textup{L}^q(\T^N)}
\le
\left(
\sum_{j=0}^{N}
\bigl(\log|A_j|^2\bigr)^{p/2}
\right)^{1/p}.
\end{equation}
\end{corollary}

The proof of Theorem~\ref{thm:cube} is based on the following inequality.

\begin{proposition}\label{prop:twofactor}
Let $1<p<2<q<\infty$ satisfy \eqref{eq:range}.   For $j=0,1$, let
\[
G_j=
\begin{pmatrix}
A_j&B_j\\
\overline{B_j}&\overline{A_j}
\end{pmatrix}
\in\SU(1,1).
\]
Then
\begin{equation}\label{eq:twofactor}
\left(\int_{\T}\bigl(\log|A_0A_1+B_0\overline{B_1}e^{-2\pi i t}|^2\bigr)^{q/2}\dd t\right)^{1/q}
\le
\left(\bigl(\log|A_0|^2\bigr)^{p/2}+\bigl(\log|A_1|^2\bigr)^{p/2}\right)^{1/p}.
\end{equation}
\end{proposition}

The logarithm on the left-hand side is nonnegative, since its argument is the squared modulus of the upper-left entry of an $\SU(1,1)$ matrix. 

The case \(q=p/(p-1)\) of the two-factor estimate also follows from Suragan’s discrete theorem \cite{Suragan2026}, applied to a sequence supported on \(\{0,1\}\), after absorbing the phases by a translation in \(t\) as in the proof below. Iterating this estimate as in the proof of Theorem~\ref{thm:cube} and using monotonicity of \(\textup L^q\)-norms on the normalized torus gives Theorem~\ref{thm:cube} throughout \(2<q\le p/(p-1)\). The remaining admissible range satisfies \(1/p+1/q<1\).

\begin{example}
Let $d=2$, and let $G_{00},G_{01},G_{10},G_{11}\in\SU(1,1)$. Define
\[
a_\sigma(t_1):=A_{\sigma 0}A_{\sigma 1}+B_{\sigma 0}\overline{B_{\sigma 1}}e^{-2\pi i t_1},
\quad
b_\sigma(t_1):=A_{\sigma 0}B_{\sigma 1}e^{2\pi i t_1}+B_{\sigma 0}\overline{A_{\sigma 1}},
\]
for $\sigma\in\cube$. Then
\[
P_\sigma^{(1)}(t_1)=
\begin{pmatrix}
 a_\sigma(t_1) & b_\sigma(t_1)\\
 \overline{b_\sigma(t_1)} & \overline{a_\sigma(t_1)}
\end{pmatrix},
\quad \sigma\in\cube,
\]
and the root matrix
\[
P_{\emptyset}^{(2)}(t_1,t_2)=P_0^{(1)}(t_1)\,D(t_2)\,P_1^{(1)}(t_1)\,D(t_2)^{-1}
\]
has
\begin{equation*}
a_{\emptyset}^{(2)}(t_1,t_2)=a_0(t_1)a_1(t_1)+b_0(t_1)\overline{b_1(t_1)}e^{-2\pi i t_2}.
\end{equation*}
Consequently, Theorem~\ref{thm:cube} gives
\begin{equation}\label{eq:depth-two-ineq}
\left\|\bigl(\log|a_{\emptyset}^{(2)}|^2\bigr)^{1/2}\right\|_{\textup{L}^q(\T^2)}
\le
\left(
\sum_{\omega\in\cube^2}\bigl(\log|A_\omega|^2\bigr)^{p/2}
\right)^{1/p}.
\end{equation}
In the setting of Proposition~\ref{prop:lin-cube}, dividing the left-hand side of \eqref{eq:depth-two-ineq} by \(\epsilon\) and letting \(\epsilon\downarrow0\) gives
\[
\left\|c_{00}+c_{01}e^{2\pi i t_1}+c_{10}e^{2\pi i t_2}+c_{11}e^{2\pi i(t_1+t_2)}\right\|_{\textup{L}^q(\T^2)},
\]
which, after the reflection $(t_1,t_2)\mapsto(-t_1,-t_2)$, is the left-hand side of the $d=2$ instance of the  inequality from \cite{CKS25}.
\end{example}

\section{Proofs of the main estimates}

\subsection{Proof of Proposition~\ref{prop:twofactor}} We first recall the rearrangement principle used in the logarithmic comparison. This is the Hardy--Littlewood--P\'{o}lya theorem in integral form \cite{HLP}. The result can also be viewed as a continuous variant of Karamata's inequality \cite{Karamata32}.

\begin{lemma}\label{lem:hlp}
Let $f,g\in \textup{L}^1([0,1])$, and let $f^*,g^*$ be their decreasing rearrangements.  Assume that
\[
\int_0^a f^*(x)\dd x\le \int_0^a g^*(x)\dd x\qquad(0\le a\le1),
\]
and
\[
\int_0^1 f^*(x)\dd x=\int_0^1 g^*(x)\dd x.
\]
Then for every convex function $\Phi\colon\R\to\R$ such that $\Phi\circ f,\Phi\circ g\in \textup{L}^1([0,1])$ one has
\[
\int_0^1\Phi(f(t))\dd t\le \int_0^1\Phi(g(t))\dd t.
\]
\end{lemma}

The next lemma is the only nonlinear comparison in the proof. It compares the logarithmic expression associated with a two-factor $\SU(1,1)$ product with the squared modulus of a two-term trigonometric polynomial. 

\begin{lemma}\label{lem:maj}
Let $u,v\in[0,\infty)$, and set
\[
A:=\log\cosh u,\quad B:=\log\cosh v,\quad
c:=\cosh u\cosh v,\quad d:=\sinh u\sinh v.
\]
Define
\[
f(t):=\log|c+de^{-2\pi i t}|,
\quad
g(t):=A+B+2\sqrt{AB}\cos(2\pi t),
\quad t\in[0,1].
\]
Then $f$ is majorized by $g$ on $[0,1]$.  Consequently, for every $r\ge1$,
\begin{equation}\label{eq:maj-consequence}
\int_0^1 f(t)^r\dd t\le \int_0^1 g(t)^r\dd t.
\end{equation}
\end{lemma}

\begin{proof}
Since $c-d=\cosh(u-v)\ge1$, one has $f\ge0$.  Also $g\ge(\sqrt A-\sqrt B)^2\ge0$.  Write
\[
\Lambda(s):=\frac12\log(c^2+d^2+2cds),
\quad
L(s):=A+B+2\sqrt{AB}\,s,
\quad -1\le s\le1.
\]
Then $f(t)=\Lambda(\cos2\pi t)$ and $g(t)=L(\cos2\pi t)$.  The function $\Lambda$ is increasing and concave, while $L$ is increasing and affine.

Set $\psi(x):=\sqrt{\log\cosh x}$ for $x\ge0$.  For $x>0$, a direct computation gives
\[
\psi''(x)=\frac{2\log\cosh x-\sinh^2x}{4(\log\cosh x)^{3/2}\cosh^2x}\le0,
\]
because $2\log\cosh x=\log(1+\sinh^2x)\le\sinh^2x$.  Hence $\psi$ is concave, and since $\psi(0)=0$, it is subadditive:
\begin{equation}\label{eq:psi-subadditive}
\psi(x+y)\le\psi(x)+\psi(y),\quad x,y\ge0.
\end{equation}
It also follows that $|\psi(x)-\psi(y)|\le\psi(|x-y|)$.

Let $H:=\Lambda-L$.  The function $H$ is concave on $[-1,1]$.  At the endpoints,
\[
H(1)=\psi(u+v)^2-(\psi(u)+\psi(v))^2\le0
\]
by \eqref{eq:psi-subadditive}, and
\[
H(-1)=\psi(|u-v|)^2-(\psi(u)-\psi(v))^2\ge0.
\]
Moreover $H(0)\ge0$, since
\[
H(0)=\frac12\log(c^2+d^2)-(A+B)\ge\frac12\log(c^2)-(A+B)=0.
\]
Thus there is $s_0\in[0,1]$ such that $H(s)\le0$ for $s\in[s_0,1]$ and $H(s)\ge0$ for $s\in[-1,s_0]$.

The means of $f$ and $g$ agree.  Indeed, with $\rho=d/c\in[0,1)$,
\[
f(t)=\log c+\log|1+\rho e^{-2\pi i t}|,
\]
and the second term has mean zero on $[0,1]$.  Hence $\int_0^1f=A+B=\int_0^1g$.

Since $\cos(2\pi t)$ and $\cos(\pi x)$ are equimeasurable on $[0,1]$, and since $\Lambda$ and $L$ are increasing,
\[
f^*(x)=\Lambda(\cos\pi x),\qquad g^*(x)=L(\cos\pi x),\qquad0\le x\le1.
\]
With $x_0=\pi^{-1}\arccos s_0$, the difference $f^*-g^*$ is nonpositive on $[0,x_0]$ and nonnegative on $[x_0,1]$, while its integral over $[0,1]$ is zero.  Therefore
\[
\int_0^a f^*(x)\dd x\le\int_0^a g^*(x)\dd x\qquad(0\le a\le1),
\]
which is the desired majorization.  Lemma~\ref{lem:hlp}, applied to $\Phi(x)=x^r$, gives \eqref{eq:maj-consequence}.
\end{proof}

The remaining result we need is the sharp inequality from \cite[Lemma~4]{CKS25}.

\begin{lemma}\label{lem:cks}
Let $1<p<2<q<\infty$ satisfy \eqref{eq:range}.  Then for all $\alpha,\beta\in[0,\infty)$,
\begin{equation*}
\left(\int_0^1|\alpha+\beta e^{-2\pi i t}|^q\dd t\right)^{1/q}
\le(\alpha^p+\beta^p)^{1/p}.
\end{equation*}
\end{lemma}

\begin{proof}[Proof of Proposition~\ref{prop:twofactor}]
Choose $u_j\ge0$ and phases such that
\[
A_j=e^{i\alpha_j}\cosh u_j,
\qquad
B_j=e^{i\beta_j}\sinh u_j,
\qquad j=0,1.
\]
After a translation in $t$,
\[
|A_0A_1+B_0\overline{B_1}e^{-2\pi i t}|
=|c+de^{-2\pi i t}|,
\]
where $c=\cosh u_0\cosh u_1$ and $d=\sinh u_0\sinh u_1$.  Put
\[
A:=\log\cosh u_0,
\qquad B:=\log\cosh u_1,
\qquad r:=q/2.
\]
By Lemma~\ref{lem:maj} we have
\[
\int_0^1\bigl(\log|c+de^{-2\pi i t}|\bigr)^{q/2}\dd t
\le
\int_0^1\bigl(A+B+2\sqrt{AB}\cos2\pi t\bigr)^{q/2}\dd t.
\]
For $\alpha:=\sqrt{2A}$ and $\beta:=\sqrt{2B}$,
\[
A+B+2\sqrt{AB}\cos(2\pi t)=\frac12|\alpha+\beta e^{-2\pi i t}|^2.
\]
Lemma~\ref{lem:cks} gives
\[
\int_0^1\bigl(A+B+2\sqrt{AB}\cos2\pi t\bigr)^{q/2}\dd t
\le
2^{-q/2}(\alpha^p+\beta^p)^{q/p}.
\]
Since $\alpha^p=(2A)^{p/2}$ and $\beta^p=(2B)^{p/2}$, this is
\[
\bigl(A^{p/2}+B^{p/2}\bigr)^{q/p}.
\]
Multiplying by the factor $2^{q/2}$ coming from $\log|\cdot|^2=2\log|\cdot|$, and using $\log|A_j|^2=2\log\cosh u_j$, yields \eqref{eq:twofactor}.
\end{proof}

\subsection{Proofs of Theorem~\ref{thm:cube} and Corollary~\ref{cor:chain}}

The proof of Theorem~\ref{thm:cube} uses Proposition~\ref{prop:twofactor} and the following 
iteration.

\begin{lemma}\label{lem:tree-iterate}
Let $1\le p\le q<\infty$, let $d\in\N$, and let
\[
X_\sigma^{(m)}\colon\T^m\to[0,\infty),
\qquad \sigma\in\cube^{d-m},
\qquad0\le m\le d,
\]
be measurable functions.  Assume that, for every $1\le m\le d$ and every $\sigma\in\cube^{d-m}$,
\begin{equation}\label{eq:tree-iterate-hyp}
\left\|X_\sigma^{(m)}(t',\cdot)\right\|_{\textup L^q(\T)}
\le
\left(X_{\sigma0}^{(m-1)}(t')^p+X_{\sigma1}^{(m-1)}(t')^p\right)^{1/p}
\end{equation}
for almost every $t'\in\T^{m-1}$.  Suppose also that the leaves are constant,
\[
X_\omega^{(0)}=x_\omega\in[0,\infty),\qquad \omega\in\cube^d.
\]
Then for every $0\le m\le d$ and every $\sigma\in\cube^{d-m}$,
\begin{equation}\label{eq:tree-iterate-concl}
\left\|X_\sigma^{(m)}\right\|_{\textup L^q(\T^m)}
\le
\left(\sum_{\tau\in\cube^m}x_{\sigma\tau}^p\right)^{1/p}.
\end{equation}
\end{lemma}

\begin{proof}
The proof is by induction on $m$.  The case $m=0$ is tautological.  Assume \eqref{eq:tree-iterate-concl} holds at depth $m-1$.  Set $r=q/p\ge1$.  Raising \eqref{eq:tree-iterate-hyp} to the power $p$, integrating in $t'$, and applying Minkowski's inequality in $\textup L^r(\T^{m-1})$, we get
\begin{align*}
\left\|X_\sigma^{(m)}\right\|_{\textup L^q(\T^m)}^p
&\le
\left\|X_{\sigma0}^{(m-1)\,p}+X_{\sigma1}^{(m-1)\,p}\right\|_{\textup L^r(\T^{m-1})}\\
&\le
\left\|X_{\sigma0}^{(m-1)}\right\|_{\textup L^q(\T^{m-1})}^p+
\left\|X_{\sigma1}^{(m-1)}\right\|_{\textup L^q(\T^{m-1})}^p.
\end{align*}
The induction hypothesis gives the desired sum over the two descendant subtrees.
\end{proof}

\begin{proof}[Proof of Theorem~\ref{thm:cube}]
For $0\le n\le d$ and $\sigma\in\cube^{d-n}$ define
\[
X_\sigma^{(n)}:=\bigl(\log|a_\sigma^{(n)}|^2\bigr)^{1/2},
\qquad
x_\omega:=\bigl(\log|A_\omega|^2\bigr)^{1/2}.
\]
The recursion \eqref{eq:tree-recursion} gives, for $t'=(t_1,\dots,t_{n-1})$ and $s=t_n$,
\[
a_\sigma^{(n)}(t',s)=a_{\sigma0}^{(n-1)}(t')a_{\sigma1}^{(n-1)}(t')+
 b_{\sigma0}^{(n-1)}(t')\overline{b_{\sigma1}^{(n-1)}(t')}e^{-2\pi i s}.
\]
Since the two child matrices lie in $\SU(1,1)$, Proposition~\ref{prop:twofactor} applies pointwise in $t'$.  Thus the hypotheses of Lemma~\ref{lem:tree-iterate} are satisfied, and \eqref{eq:cube-main} is its conclusion at the root.
\end{proof}

\begin{proof}[Proof of Corollary~\ref{cor:chain}]
Let $0^N\in\cube^N$ be the all-zero word, and let $e_k\in\cube^N$ be the word with a single $1$ in the $k$th position.  Define leaves
\[
H_\omega:=
\begin{cases}
G_0, & \omega=0^N,\\
G_j, & \omega=e_{N-j+1}\text{ for some }1\le j\le N,\\
I, & \text{otherwise}.
\end{cases}
\]
Let $R^{(n)}$ denote the root product generated by this sparse pattern from $G_0,\dots,G_n$.  We claim that
\begin{equation}\label{eq:sparse-tree-identity}
R^{(n)}(t_1,\dots,t_n)=G_0\wt G_1(t_1)\cdots\wt G_n(t_n).
\end{equation}
This follows by induction on $n$: the left subtree at level $n$ is the sparse tree for $G_0,\dots,G_{n-1}$, and the right subtree contains only the leaf $G_n$ at its all-zero position, hence has root $G_n$.

Applying Theorem~\ref{thm:cube} to the leaves $H_\omega$ and using \eqref{eq:sparse-tree-identity}, all identity leaves contribute $0$ to the right-hand side.  The upper-left entries of the remaining leaves are $A_0,\dots,A_N$, which proves \eqref{eq:chain}.
\end{proof}

\section{Linearization and optimality}

We use the notation
\[
\omega\cdot t:=\omega_1t_d+\omega_2t_{d-1}+\cdots+\omega_dt_1,
\qquad \omega=(\omega_1,\dots,\omega_d)\in\cube^d.
\]
This reversed pairing of binary digits and torus variables is forced by the tree recursion; it differs from the usual pairing only by a permutation of coordinates.

\begin{proposition}\label{prop:lin-cube}
Fix $d\in\N$ and complex numbers $c_\omega$, $\omega\in\cube^d$.  For $\epsilon>0$, set
\[
G_\omega(\epsilon)=
\begin{pmatrix}
\sqrt{1+\epsilon^2|c_\omega|^2}&\epsilon c_\omega\\
\epsilon\overline{c_\omega}&\sqrt{1+\epsilon^2|c_\omega|^2}
\end{pmatrix}\in\SU(1,1).
\]
Let $P_\sigma^{(n)}(\cdot;\epsilon)$ be the associated recursive products, and let $a_\sigma^{(n)}(\cdot;\epsilon)$ denote their upper-left entries.  Then, uniformly in $t\in\T^d$,
\begin{equation}\label{eq:lin-cube}
\log|a_\emptyset^{(d)}(t;\epsilon)|^2
=
\epsilon^2\left|\sum_{\omega\in\cube^d}c_\omega e^{2\pi i\omega\cdot t}\right|^2+O(\epsilon^4),
\end{equation}
and, for each $\omega\in\cube^d$,
\begin{equation}\label{eq:lin-leaf}
\log|A_\omega(\epsilon)|^2=\epsilon^2|c_\omega|^2+O(\epsilon^4).
\end{equation}
\end{proposition}

\begin{proof}
Write
\[
P_\sigma^{(n)}(t;\epsilon)=
\begin{pmatrix}
a_\sigma^{(n)}(t;\epsilon)&b_\sigma^{(n)}(t;\epsilon)\\
g_\sigma^{(n)}(t;\epsilon)&\overline{a_\sigma^{(n)}(t;\epsilon)}
\end{pmatrix},
\qquad g_\sigma^{(n)}=\overline{b_\sigma^{(n)}}.
\]
We prove by induction that
\begin{equation}\label{eq:lin-claim-a}
a_\sigma^{(n)}(t;\epsilon)=1+O(\epsilon^2)
\end{equation}
and
\begin{equation}\label{eq:lin-claim-g}
g_\sigma^{(n)}(t;\epsilon)=\epsilon F_\sigma^{(n)}(t)+O(\epsilon^3),
\end{equation}
uniformly in the relevant variables, where
\[
F_\omega^{(0)}:=\overline{c_\omega},
\]
and
\begin{equation}\label{eq:F-recursion}
F_\sigma^{(n)}(t_1,\dots,t_n):=
F_{\sigma0}^{(n-1)}(t_1,\dots,t_{n-1})+
 e^{-2\pi i t_n}F_{\sigma1}^{(n-1)}(t_1,\dots,t_{n-1}).
\end{equation}
The case $n=0$ is immediate.  If the claims hold at level $n-1$, then the lower-left entry of
\[
P_{\sigma0}^{(n-1)}(t';\epsilon)D(s)P_{\sigma1}^{(n-1)}(t';\epsilon)D(s)^{-1}
\]
is
\[
g_{\sigma0}^{(n-1)}(t';\epsilon)a_{\sigma1}^{(n-1)}(t';\epsilon)+
\overline{a_{\sigma0}^{(n-1)}(t';\epsilon)}g_{\sigma1}^{(n-1)}(t';\epsilon)e^{-2\pi i s},
\]
which gives \eqref{eq:lin-claim-g}.  Formula \eqref{eq:lin-claim-a} follows from the corresponding expression for the $(1,1)$-entry.

Iterating \eqref{eq:F-recursion},
\[
F_\emptyset^{(d)}(t)=\sum_{\omega\in\cube^d}\overline{c_\omega}e^{-2\pi i\omega\cdot t}
=
\overline{\sum_{\omega\in\cube^d}c_\omega e^{2\pi i\omega\cdot t}}.
\]
Since every recursive matrix lies in $\SU(1,1)$,
\[
|a_\emptyset^{(d)}(t;\epsilon)|^2=1+|g_\emptyset^{(d)}(t;\epsilon)|^2.
\]
Combining this with \eqref{eq:lin-claim-g} gives \eqref{eq:lin-cube}.  The leaf expansion \eqref{eq:lin-leaf} follows from
\[
\log|A_\omega(\epsilon)|^2=\log(1+\epsilon^2|c_\omega|^2).\qedhere
\]
\end{proof}

The proposition above shows that Theorem~\ref{thm:cube} linearizes to the sharp binary-cube Hausdorff--Young inequality of \cite[Theorem~1]{CKS25}. More precisely, let $1<p<2<q<\infty$ satisfy \eqref{eq:range}. Then, for every $d\in\N$ and all coefficients $c_\omega\in\C$,
\begin{equation}\label{eq:full-cks}
\left\|\sum_{\omega\in\cube^d}c_\omega e^{-2\pi i\omega\cdot t}\right\|_{\textup L^q(\T^d)}
\le
\left(\sum_{\omega\in\cube^d}|c_\omega|^p\right)^{1/p}.
\end{equation}
To obtain this inequality, apply Theorem~\ref{thm:cube} to the matrices $G_\omega(\epsilon)$ from Proposition~\ref{prop:lin-cube}, divide by $\epsilon$, and let $\epsilon\downarrow0$. The uniform expansion in \eqref{eq:lin-cube}, together with $|\sqrt{x}-\sqrt{y}|\le\sqrt{|x-y|}$ for $x,y\ge0$, gives the limiting left-hand side. The leaf expansion \eqref{eq:lin-leaf} gives the limiting right-hand side. Finally, the reflection $t\mapsto-t$ preserves Haar measure on $\T^d$ and yields the displayed sign in the exponential.

The exponent range in Theorem~\ref{thm:cube} is optimal. Indeed, if the nonlinear inequality held with constant $1$ for some $1<p<2<q<\infty$ not satisfying \eqref{eq:range}, the same linearization argument would give \eqref{eq:full-cks} for those exponents, contrary to \cite[Theorem~1]{CKS25}. The same range is optimal for Corollary~\ref{cor:chain}, already by the two-factor case $N=1$, whose linearization is the one-dimensional two-point inequality. The necessity of \eqref{eq:range} already holds in this case; see \cite[Section~3.1]{CKS25}.

\section{Alternating chains and additive energies}
This section records consequences of the nonlinear theorem. We apply Theorem~\ref{thm:cube} to generating functions for alternating
chains in the binary cube. Throughout this section we order $\cube^d$  lexicographically and write
\[
\wt G_\omega(t):=D(\omega\cdot t)G_\omega D(\omega\cdot t)^{-1}.
\]

\begin{lemma}\label{lem:flatten}
The root matrix in \eqref{eq:tree-recursion} satisfies 
\begin{equation*}
P_\emptyset^{(d)}(t)=\prod_{\omega\in\cube^d}^{\mathrm{lex}}\wt G_\omega(t).
\end{equation*}
\end{lemma}

\begin{proof}
For $d=1$, the formula is $G_0D(t_1)G_1D(t_1)^{-1}=\wt G_0(t)\wt G_1(t)$.  The induction step follows by applying the formula to the left and right subtrees.  Since all diagonal matrices commute,
\[
D(t_d)D(\eta\cdot t')=D(t_d+\eta\cdot t'),
\]
and the right subtree contributes the lexicographic block with leading digit $1$.
\end{proof}

\begin{proposition}\label{prop:alt-expansion}
Let $E\subseteq\cube^d$, let $0\le r<1$, and assign to each leaf the matrix
\[
G_\omega(r):=
\begin{cases}
(1-r^2)^{-1/2}\begin{pmatrix}1&r\\ r&1\end{pmatrix}, & \omega\in E,\\[1.1ex]
I, & \omega\notin E.
\end{cases}
\]
Let $a_E(t;r)$ be the $(1,1)$-entry of the root matrix.  For $k\ge0$ and $\xi\in\Z^d$, define
\[
N_{E,k}(\xi):=
\#\Bigl\{(\omega_1,\dots,\omega_{2k})\in E^{2k}:\ \omega_1<\cdots<\omega_{2k},\
\omega_2-\omega_1+\cdots+\omega_{2k}-\omega_{2k-1}=\xi\Bigr\},
\]
with $N_{E,0}(0)=1$ and $N_{E,0}(\xi)=0$ for $\xi\ne0$.  Then
\begin{equation}\label{eq:alt-expansion}
a_E(t;r)=(1-r^2)^{-|E|/2}
\sum_{k=0}^{\lfloor |E|/2\rfloor}r^{2k}\sum_{\xi\in\Z^d}N_{E,k}(\xi)e^{-2\pi i\xi\cdot t}.
\end{equation}
\end{proposition}

\begin{proof}
By Lemma~\ref{lem:flatten}, we may omit the identity factors and write
\[
P_\emptyset^{(d)}(t)=(1-r^2)^{-|E|/2}
\prod_{\omega\in E}^{\mathrm{lex}}
\left(I+rX_\omega(t)\right),
\]
where
\[
X_\omega(t):=\begin{pmatrix}0&e^{2\pi i\omega\cdot t}\\ e^{-2\pi i\omega\cdot t}&0\end{pmatrix}.
\]
The $(1,1)$-entry receives contributions only from even products.  For $\omega_1<\cdots<\omega_{2k}$,
\[
\bigl(X_{\omega_1}(t)\cdots X_{\omega_{2k}}(t)\bigr)_{11}
=e^{-2\pi i(\omega_2-\omega_1+\cdots+\omega_{2k}-\omega_{2k-1})\cdot t}.
\]
Collecting equal frequencies gives \eqref{eq:alt-expansion}.
\end{proof}

The expansion \eqref{eq:alt-expansion} is related to the generating
functions for alternating ordered partitions of integers studied
by Saksida \cite{Saksida24}. Here the alternating sums are formed
from lexicographically ordered points of $E\subseteq\cube^d$
and take values in $\Z^d$.

Let $1<p<2<q<\infty$ satisfy \eqref{eq:range}, and let $E\subseteq\cube^d$. Applying Theorem~\ref{thm:cube} to the leaf family in the above proposition gives, for \(0\le r<1\),  
\begin{equation}\label{eq:alt-nonlinear}
\left\|\bigl(\log|a_E(\cdot;r)|^2\bigr)^{1/2}\right\|_{\textup L^q(\T^d)}
\le
|E|^{1/p}\left(\log\frac1{1-r^2}\right)^{1/2}.
\end{equation}

\noindent Indeed, each marked leaf contributes \(\log(1-r^2)^{-1}\), and every unmarked leaf contributes \(0\).

\begin{corollary}
Let $m\ge2$ be an integer and set
\begin{equation*}
p_m:=\frac{2m}{\log_2\binom{2m}{m}}.
\end{equation*}
For $E\subseteq\cube^d$ and $0\le r<1$, define
\[
\mu_{E,r}(\xi):=\sum_{k=0}^{\lfloor |E|/2\rfloor}r^{2k}N_{E,k}(\xi),
\qquad \xi\in\Z^d.
\]
Then
\begin{equation}\label{eq:centered-weighted-energy}
\int_{\T^d}
\left(
\left|\sum_{\xi\in\Z^d}\mu_{E,r}(\xi)e^{-2\pi i\xi\cdot t}\right|^2-(1-r^2)^{|E|}
\right)^m\dd t
\le
C_{m,|E|}(r)|E|^{2m/p_m},
\end{equation}
where
\begin{equation*}
C_{m,n}(r):=(1-r^2)^{mn}
\left(
\frac{\left(\frac{1+r}{1-r}\right)^n-1}{\log\left(\left(\frac{1+r}{1-r}\right)^n\right)}
\log\frac1{1-r^2}
\right)^m.
\end{equation*}
When both the numerator and denominator in the fraction vanish, the fraction is interpreted as $1$.
\end{corollary}

\begin{proof}
Take $q=2m$ and $p=p_m$ in \eqref{eq:alt-nonlinear}.  Let $n:=|E|$.  The cases $n=0$ and $r=0$ are trivial, so assume $n\ge1$ and $0<r<1$.  By Proposition~\ref{prop:alt-expansion} we have
\[
\widehat\mu_{E,r}(t):=\sum_{\xi\in\Z^d}\mu_{E,r}(\xi)e^{-2\pi i\xi\cdot t}
=(1-r^2)^{n/2}a_E(t;r).
\]
Every marked factor in the flattened product is unitarily conjugate to
\[
H_r=(1-r^2)^{-1/2}\begin{pmatrix}1&r\\ r&1\end{pmatrix},
\]
whose operator norm is $((1+r)/(1-r))^{1/2}$. Since the root matrix lies in $\SU(1,1)$, we have
$|a_E(t;r)|^2\ge1$. On the other hand, submultiplicativity
of the operator norm gives
\[
|a_E(t;r)|^2\le M_n(r):=
\left(\frac{1+r}{1-r}\right)^n.
\] 
The concavity of $\log$ on $[1,M_n(r)]$ now gives
\[
x-1\le \frac{M_n(r)-1}{\log M_n(r)}\log x,
\qquad 1\le x\le M_n(r).
\]

Apply this inequality with $x=|a_E(t;r)|^2$ and raise both sides
to the power $m$. Then integrate over $\T^d$ and use
\eqref{eq:alt-nonlinear}. Since
\[
|\widehat\mu_{E,r}(t)|^2-(1-r^2)^n
=(1-r^2)^n\bigl(|a_E(t;r)|^2-1\bigr),
\]
we obtain \eqref{eq:centered-weighted-energy}.
\end{proof}

Finally, let $m\ge2$ be an integer.  For $E\subseteq\cube^d$, define
\begin{equation}\label{eq:additive-energy}
\mathsf E_m(E):=
\#\Bigl\{(x_1,\dots,x_m,y_1,\dots,y_m)\in E^{2m}:\ x_1+\cdots+x_m=y_1+\cdots+y_m\Bigr\}.
\end{equation}
Then
\begin{equation*}
\mathsf E_m(E)\le |E|^{2m/p_m}.
\end{equation*}

Indeed, as $r\downarrow0$, we have
\[
\widehat\mu_{E,r}(t)=1+r^2\sum_{\omega_1<\omega_2\in E}e^{-2\pi i(\omega_2-\omega_1)\cdot t}+O(r^4),
\]
uniformly in $t$. Hence,
\[
|\widehat\mu_{E,r}(t)|^2-(1-r^2)^{|E|}
=r^2\left|\sum_{\omega\in E}e^{-2\pi i\omega\cdot t}\right|^2+O(r^4).
\]
Also $C_{m,|E|}(r)=r^{2m}(1+O(r))$.  Dividing the estimate \eqref{eq:centered-weighted-energy} by $r^{2m}$ and letting $r\downarrow0$, we obtain
\[
\int_{\T^d}\left|\sum_{\omega\in E}e^{-2\pi i\omega\cdot t}\right|^{2m}\dd t
\le |E|^{2m/p_m}.
\]
By orthogonality of characters, the integral on the left equals
$\mathsf E_m(E)$. This sharp bound is already contained in the
linear theorem \cite{CKS25}. Here we recover it from
\eqref{eq:centered-weighted-energy} by letting $r\downarrow0$.

\section*{Declaration of AI usage}
During the preparation of this manuscript, ChatGPT 5.4 Pro and Gemini 3.1 Pro were used to explore research directions, suggest intermediate lemmas, draft portions of the exposition, proofread the text, and refine the language. However, the original concepts, main ideas for the proofs, proof optimization, and the foundational writing of the manuscript are entirely the work of the authors. The authors independently verified all mathematical claims, proofs, and references, and assume full responsibility for the final content.

\section*{Acknowledgments and funding}
Both authors were supported in part by the Croatian Science Foundation under the project number HRZZ-IP-2022-10-5116 (FANAP). T.C. was supported in part by the European Union - NextGenerationEU, project: IP-UNIST-44 (ITPEM).

\bibliography{multi_nonlin_hy_SU}{}

@article {CK01a,
    AUTHOR = {Christ, Michael and Kiselev, Alexander},
     TITLE = {Maximal functions associated to filtrations},
   JOURNAL = {J. Funct. Anal.},
  FJOURNAL = {Journal of Functional Analysis},
    VOLUME = {179},
      YEAR = {2001},
    NUMBER = {2},
     PAGES = {409--425},
      ISSN = {0022-1236,1096-0783},
   MRCLASS = {47B38 (47H30)},
  MRNUMBER = {1809116},
MRREVIEWER = {Isaac\ V.\ Shragin},
       DOI = {10.1006/jfan.2000.3687},
       URL = {https://doi.org/10.1006/jfan.2000.3687},
}

@article {CK01b,
    AUTHOR = {Christ, Michael and Kiselev, Alexander},
     TITLE = {W{KB} asymptotic behavior of almost all generalized
              eigenfunctions for one-dimensional {S}chr\"odinger operators
              with slowly decaying potentials},
   JOURNAL = {J. Funct. Anal.},
  FJOURNAL = {Journal of Functional Analysis},
    VOLUME = {179},
      YEAR = {2001},
    NUMBER = {2},
     PAGES = {426--447},
      ISSN = {0022-1236,1096-0783},
   MRCLASS = {34E20 (34L10 81Q20)},
  MRNUMBER = {1809117},
       DOI = {10.1006/jfan.2000.3688},
       URL = {https://doi.org/10.1006/jfan.2000.3688},
}

@article {CKS25,
    AUTHOR = {Crmari\'{c}, Ton\'{c}i and Kova\v{c}, Vjekoslav and Shiraki, Shobu},
     TITLE = {Inequalities in {F}ourier analysis on binary cubes},
   JOURNAL = {Discrete Anal.},
  FJOURNAL = {Discrete Analysis},
       URL = {https://arxiv.org/abs/2507.01359},
      YEAR = {2026},
      NOTE = {To appear}
}

@book {HLP,
    AUTHOR = {Hardy, G. H. and Littlewood, J. E. and P\'olya, G.},
     TITLE = {Inequalities},
      NOTE = {2d ed},
 PUBLISHER = {Cambridge, at the University Press},
      YEAR = {1952},
     PAGES = {xii+324},
   MRCLASS = {27.0X},
  MRNUMBER = {46395},
}

@article {Karamata32,
    AUTHOR = {Karamata, Jovan},
     TITLE = {Sur une in{\'{e}}galit{\'{e}} relative aux fonctions convexes},
   JOURNAL = {Publ. Math. Univ. Belgrade},
    VOLUME = {1},
      YEAR = {1932},
     PAGES = {145--148},
}

@article {K12,
    AUTHOR = {Kova\v{c}, Vjekoslav},
     TITLE = {Uniform constants in {H}ausdorff-{Y}oung inequalities for the
              {C}antor group model of the scattering transform},
   JOURNAL = {Proc. Amer. Math. Soc.},
  FJOURNAL = {Proceedings of the American Mathematical Society},
    VOLUME = {140},
      YEAR = {2012},
    NUMBER = {3},
     PAGES = {915--926},
      ISSN = {0002-9939,1088-6826},
   MRCLASS = {34L25 (42A38)},
  MRNUMBER = {2869075},
MRREVIEWER = {Petru\ A.\ Cojuhari},
       DOI = {10.1090/S0002-9939-2011-11078-5},
       URL = {https://doi.org/10.1090/S0002-9939-2011-11078-5},
}

@article {KOR19,
    AUTHOR = {Kova\v{c}, Vjekoslav and Oliveira e Silva, Diogo and Rup\v{c}i\'{c}, Jelena},
     TITLE = {A sharp nonlinear {H}ausdorff-{Y}oung inequality for small
              potentials},
   JOURNAL = {Proc. Amer. Math. Soc.},
  FJOURNAL = {Proceedings of the American Mathematical Society},
    VOLUME = {147},
      YEAR = {2019},
    NUMBER = {1},
     PAGES = {239--253},
      ISSN = {0002-9939,1088-6826},
   MRCLASS = {42A38 (34L25)},
  MRNUMBER = {3876746},
MRREVIEWER = {Valeri\ S.\ Serov},
       DOI = {10.1090/proc/14268},
       URL = {https://doi.org/10.1090/proc/14268},
}

@article {KOR22,
    AUTHOR = {Kova\v{c}, Vjekoslav and Oliveira e Silva, Diogo and Rup\v{c}i\'{c}, Jelena},
     TITLE = {Asymptotically sharp discrete nonlinear {H}ausdorff-{Y}oung
              inequalities for the {$\rm SU(1,1)$}-valued {F}ourier
              products},
   JOURNAL = {Q. J. Math.},
  FJOURNAL = {The Quarterly Journal of Mathematics},
    VOLUME = {73},
      YEAR = {2022},
    NUMBER = {3},
     PAGES = {1179--1188},
      ISSN = {0033-5606,1464-3847},
   MRCLASS = {42A05 (42C05 43A25)},
  MRNUMBER = {4479843},
       DOI = {10.1093/qmath/haac011},
       URL = {https://doi.org/10.1093/qmath/haac011},
}

@article {MTT03,
    AUTHOR = {Muscalu, Camil and Tao, Terence and Thiele, Christoph},
     TITLE = {A {C}arleson type theorem for a {C}antor group model of the
              scattering transform},
   JOURNAL = {Nonlinearity},
  FJOURNAL = {Nonlinearity},
    VOLUME = {16},
      YEAR = {2003},
    NUMBER = {1},
     PAGES = {219--246},
      ISSN = {0951-7715,1361-6544},
   MRCLASS = {34L25 (34A55 47E05)},
  MRNUMBER = {1950785},
       DOI = {10.1088/0951-7715/16/1/314},
       URL = {https://doi.org/10.1088/0951-7715/16/1/314},
}

@article {OSTTW12,
    AUTHOR = {Oberlin, Richard and Seeger, Andreas and Tao, Terence and
              Thiele, Christoph and Wright, James},
     TITLE = {A variation norm {C}arleson theorem},
   JOURNAL = {J. Eur. Math. Soc. (JEMS)},
  FJOURNAL = {Journal of the European Mathematical Society (JEMS)},
    VOLUME = {14},
      YEAR = {2012},
    NUMBER = {2},
     PAGES = {421--464},
      ISSN = {1435-9855,1435-9863},
   MRCLASS = {42A20 (42A16 42A45 42A61)},
  MRNUMBER = {2881301},
MRREVIEWER = {Alexander\ V.\ Tovstolis},
       DOI = {10.4171/JEMS/307},
       URL = {https://doi.org/10.4171/JEMS/307},
}

@article {OeS18,
    AUTHOR = {Oliveira e Silva, Diogo},
     TITLE = {A variational nonlinear {H}ausdorff-{Y}oung inequality in the
              discrete setting},
   JOURNAL = {Math. Res. Lett.},
  FJOURNAL = {Mathematical Research Letters},
    VOLUME = {25},
      YEAR = {2018},
    NUMBER = {6},
     PAGES = {1993--2015},
      ISSN = {1073-2780,1945-001X},
   MRCLASS = {43A05},
  MRNUMBER = {3934855},
MRREVIEWER = {Krishnan\ Parthasarathy},
       DOI = {10.4310/MRL.2018.v25.n6.a15},
       URL = {https://doi.org/10.4310/MRL.2018.v25.n6.a15},
}

@article {Poltoratski24,
    AUTHOR = {Poltoratski, A.},
     TITLE = {Pointwise convergence of the non-linear {F}ourier transform},
   JOURNAL = {Ann. of Math. (2)},
  FJOURNAL = {Annals of Mathematics. Second Series},
    VOLUME = {199},
      YEAR = {2024},
    NUMBER = {2},
     PAGES = {741--793},
      ISSN = {0003-486X,1939-8980},
   MRCLASS = {42A38 (30D15 34L25)},
  MRNUMBER = {4713022},
MRREVIEWER = {Michael\ T.\ Lacey},
       DOI = {10.4007/annals.2024.199.2.4},
       URL = {https://doi.org/10.4007/annals.2024.199.2.4},
}

@article {Rupcic19,
    AUTHOR = {Rup\v{c}i\'{c}, Jelena},
     TITLE = {Convergence of lacunary {${\rm SU}(1,1)$}-valued trigonometric
              products},
   JOURNAL = {Commun. Pure Appl. Anal.},
  FJOURNAL = {Communications on Pure and Applied Analysis},
    VOLUME = {19},
      YEAR = {2020},
    NUMBER = {3},
     PAGES = {1275--1289},
      ISSN = {1534-0392,1553-5258},
   MRCLASS = {42A55 (40A20)},
  MRNUMBER = {4064031},
MRREVIEWER = {Sergey\ V.\ Astashkin},
       DOI = {10.3934/cpaa.2020062},
       URL = {https://doi.org/10.3934/cpaa.2020062},
}

@article {Saksida22,
    AUTHOR = {Saksida, Pavle},
     TITLE = {Discrete nonlinear {F}ourier transforms and their inverses},
   JOURNAL = {Inverse Problems},
  FJOURNAL = {Inverse Problems. An International Journal on the Theory and
              Practice of Inverse Problems, Inverse Methods and Computerized
              Inversion of Data},
    VOLUME = {38},
      YEAR = {2022},
    NUMBER = {8},
     PAGES = {Paper No. 085003, 22},
      ISSN = {0266-5611,1361-6420},
   MRCLASS = {65T50 (42A38)},
  MRNUMBER = {4459107},
MRREVIEWER = {Min\ Ku},
       DOI = {10.1088/1361-6420/ac73ae},
       URL = {https://doi.org/10.1088/1361-6420/ac73ae},
}

@article {Saksida24,
    AUTHOR = {Saksida, Pavle},
     TITLE = {On the beta distribution, the nonlinear {F}ourier transform
              and a combinatorial problem},
   JOURNAL = {Ars Math. Contemp.},
  FJOURNAL = {Ars Mathematica Contemporanea},
    VOLUME = {24},
      YEAR = {2024},
    NUMBER = {1},
     PAGES = {Paper No. 8, 20},
      ISSN = {1855-3966,1855-3974},
   MRCLASS = {42A38 (05A17 60E05)},
  MRNUMBER = {4640468},
       DOI = {10.26493/1855-3974.2976.f76},
       URL = {https://doi.org/10.26493/1855-3974.2976.f76},
}

@book {Simon1,
    AUTHOR = {Simon, Barry},
     TITLE = {Orthogonal polynomials on the unit circle. {P}art 1},
    SERIES = {American Mathematical Society Colloquium Publications},
    VOLUME = {54, Part 1},
      NOTE = {Classical theory},
 PUBLISHER = {American Mathematical Society, Providence, RI},
      YEAR = {2005},
     PAGES = {xxvi+466},
      ISBN = {0-8218-3446-0},
   MRCLASS = {42-02 (30C85 33C45 42C05 47B36 47N50)},
  MRNUMBER = {2105088},
MRREVIEWER = {P.\ L.\ Duren},
       DOI = {10.1090/coll054.1},
       URL = {https://doi.org/10.1090/coll054.1},
}

@book {Simon2,
    AUTHOR = {Simon, Barry},
     TITLE = {Orthogonal polynomials on the unit circle. {P}art 2},
    SERIES = {American Mathematical Society Colloquium Publications},
    VOLUME = {54, Part 2},
      NOTE = {Spectral theory},
 PUBLISHER = {American Mathematical Society, Providence, RI},
      YEAR = {2005},
     PAGES = {i--xxii and 467--1044},
      ISBN = {0-8218-3675-7},
   MRCLASS = {42-02 (30C85 33C45 42C05 47B36 47N50)},
  MRNUMBER = {2105089},
MRREVIEWER = {P.\ L.\ Duren},
       DOI = {10.1090/coll/054.2/01},
       URL = {https://doi.org/10.1090/coll/054.2/01},
}

@misc {Suragan2026,
       AUTHOR = {Durvudkhan Suragan},
        TITLE = {The nonlinear {H}ausdorff-{Y}oung inequality},
       EPRINT = {2608.15895},
ARCHIVEPREFIX = {arXiv},
 PRIMARYCLASS = {math.FA},
          URL = {https://arxiv.org/abs/2608.15895},
         YEAR = {2026},
}

@misc {TT03,
       AUTHOR = {Terence Tao and Christoph Thiele},
        TITLE = {Nonlinear {F}ourier Analysis},
       EPRINT = {1201.5129},
ARCHIVEPREFIX = {arXiv},
 PRIMARYCLASS = {math.CA},
          URL = {https://arxiv.org/abs/1201.5129},
         YEAR = {2012},
         NOTE = {Unpublished lecture notes based on lectures from 2003},
}

@article {Verblunsky36,
    AUTHOR = {Samuel Verblunsky},
     TITLE = {On {P}ositive {H}armonic {F}unctions},
   JOURNAL = {Proc. London Math. Soc. (2)},
  FJOURNAL = {Proceedings of the London Mathematical Society. Second Series},
    VOLUME = {40},
      YEAR = {1936},
     PAGES = {290--320},
      ISSN = {0024-6115},
   MRCLASS = {99-04},
  MRNUMBER = {1575824},
       DOI = {10.1112/plms/s2-40.1.290},
}
\bibliographystyle{plainurl}

\end{document}